\documentclass[11pt,a4paper]{amsart}
\usepackage[a4paper,margin=3cm]{geometry}
\newtheorem{theorem}{Theorem}

\usepackage{graphicx}
\usepackage{amsmath}
\usepackage{amssymb}
\usepackage[latin1]{inputenc}
\usepackage{chicago}
\usepackage{latexsym}
\usepackage{tabularx} 
\usepackage{booktabs}
\usepackage{rotating}
\usepackage{environ}
\usepackage{multirow}
\usepackage{caption}
\usepackage{pdflscape}
\usepackage{lipsum}
\usepackage{enumitem}

\newtheorem{assumption}[theorem]{Assumption}

\newcommand{\indep}{\perp\hskip -7pt \perp }

\newcounter{examplecounter}
\newenvironment{example}{%
   \refstepcounter{examplecounter}%
\textbf{Example \arabic{examplecounter}}%
\quad
}

\makeatletter
\newcommand{\addresseshere}{%
	\enddoc@text\let\enddoc@text\relax
}
\makeatother

\begin{document}

\title[Model misspecification and bias]{Model misspecification and bias for inverse probability weighting and doubly robust estimators}
\author[Ingeborg Waernbaum and Laura Pazzagli]{Ingeborg Waernbaum\textsuperscript{1} and Laura Pazzagli\textsuperscript{2}}
\address{\textsuperscript{1} Department of Statistics, USBE, Ume{\aa} University, Sweden \\
	 and Institute for Evaluation of Labour Market and Education Policy, IFAU, Uppsala, Sweden
	\\
	\textsuperscript{2} Division of Statistics, Department of Economics, University of Perugia, Italy\\}

\email{ingeborg.waernbaum@umu.se}

\begin{abstract}
	In the causal inference literature an estimator belonging to a class of semi-parametric estimators is called robust if it has desirable properties under the assumption that at least one of the working models is correctly specified. In this paper we propose a crude analytical approach to study the large sample bias of semi-parameteric estimators of the average causal effect when all working models are misspecified. We apply our approach to three prototypical estimators, two inverse probability weighting (IPW) estimators, using a misspecified propensity score model, and a doubly robust (DR) estimator, using misspecified models for the outcome regression and the propensity score. To analyze the question of when the use of two misspecified models are better than one we derive necessary and sufficient conditions for when the DR estimator has a smaller bias than a simple IPW estimator and when it has a smaller bias than an IPW estimator with normalized weights. 
	If the misspecificiation of the outcome model is moderate the comparisons of the biases of the IPW and DR estimators suggest that the DR estimator has a smaller bias than the IPW estimators. However, all biases include the PS-model error and we suggest that a researcher is careful when modeling the PS whenever such a model is involved. 
	
\end{abstract}
\maketitle

\addresseshere 
\newpage
\section{Introduction}
Identifying an average causal effect of a treatment with observational data requires adjustment for background variables that affect both the treatment and the outcome under study. Often parametric models are assumed for parts of the joint distribution of the treatment, outcome and background variables (covariates) and large sample properties of estimators are derived under the assumption that the parametric models are correctly specified. 

A class of semiparametric estimators called inverse probability weighting (IPW) estimators use the difference of the weighted means of the outcomes for the treatment groups as an estimator of the average causal effect, see e.g. \citeN{LD:04}. IPW estimators reweights the observed outcomes to a random sample of all potential outcomes, missing and observed, by letting each observed outcome account for itself and other individuals with similar characteristics proportionally to the probability of their outcome being observed. IPW estimators are used in applied literature \cite{kwon2015adjuvant} and their properties have been studied in the missing data and causal inference literature, see e.g. \citeN{vansteelandt2015analysis} and \citeN{seaman2013review} for reviews.

Earlier research have considered properties of IPW estimators for estimating the average causal effect under the  assumption that a parametric propensity score (PS) model is correctly specified \cite{LD:04,yao:10}. Properties of IPW estimators using different weights, often referred to as stabilized \cite{HBR:00,HR:06} or normalized \cite{HI:01,BDM:14} have been discussed together with the impact of violations to an assumption of overlap (\emph{positivity}) \cite{KT:10,petersen2010diagnosing}. 

To decrease the reliance on the choice of a parametric model of the PS an approach with doubly or multiply robust estimators has emerged, see the review by \citeN{seaman2018introduction} for an introduction to doubly robust (DR) estimators. An estimator is referred to as a DR estimator \cite{BR:05,AT:07} since it is a consistent estimator of the average causal effect if either the model for the propensity score or the outcome regression (OR) model is correct \cite{SRR:99}. The efficiency of the DR estimator is a key property and its variance has been described under correct specification of at least one of the models \cite{CTD:09,tan2010bounded}. When both models are correct the estimator reaches the semiparametric efficiency bound described in \citeN{RRZ:94}. The large sample properties of IPW estimators with standard, normalized and variance minimized weights, together with a prototypical DR estimator were studied and compared in \citeN{LD:04} under correct specification of the PS and OR models. Multiply robust estimators allow for several PS and OR models and are consistent for the true average treatment effect if any of the multiple
models is correctly specified \cite{han2013estimation}.

There are few studies on doubly or multiply robust estimators under misspecification of both (all) the PS and the OR models. \citeN{KS:07} studied and compared the performance of various DR and non-DR estimators under misspecification of both models. They concluded that many DR methods perform better than simple inverse probability weighting. However a regression-based estimator under a misspecified model was not improved upon. The paper was commented and the relevance of the results were discussed by several authors see e.g., \citeN{TD:07}, \citeN{Tan:07}, \citeN{robins2007comment}. In \citeN{IW:12} a matching estimator was compared to IPW and DR estimators under misspecification of both the PS and OR models. Here, a robustness class for the matching estimator under misspecification of the PS model was described.

In this paper we describe three commonly used semi-parametric estimators of the average causal effect under the assumption that none of the working models are correctly specified. For this purpose we study the difference between the probability limit of the estimator under model misspecification and the true average causal effect. The purpose of this definition of the bias is that the estimators under study converges to a well-defined limit however not necessarily consistent for the true average causal effect. We study the biases of two IPW estimators and a DR estimator and compare them under the same misspecification of the PS-model. In the comparisons with the DR estimator the biases provide a means to describe when two wrong models are better than one. To analyze the consequences of the model misspecifications we compare the absolute values of the biases in two parts separately, one for each of the two potential outcome means ($\mu_1,\mu_0$). For the comparisons we  provide sufficient and necessary conditions for inequalities involving the  absolute value of the biases of the different estimators. We use a running example of a data generating process with misspecified PS and OR models to illustrate the inequalities. A simulation study is performed to investigate the biases for finite samples. The data generating processes and the misspecified models from the simulation designs are also used for numerical approximations of the large sample properties derived in the paper.     

In recent studies strategies for bias reduction under model misspecification have been proposed by inclusion of additional conditions in the estimating equations for both IPW \cite {IR:14} and DR estimators \cite {VV:14}. However, the general approach for analyzing model misspecification provided in this paper could also be used to study the biased reduced estimators.

The paper proceeds as follows. Section \ref{theory} presents the model and theory together with the estimators and their properties  when the working models are correctly specified. Section \ref{general} presents a general approach and assumptions to study model misspecification. In Section \ref{BIAS} the generic biases are derived and comparisons between the estimators are performed. We present a simulation study in Section \ref{simulations} containing both finite sample properties of the estimators and numerical large sample approximations and thereafter we conclude with a discussion.   

\section{Model and theory}\label{theory}

The potential outcome framework defines a causal effect as a comparison of potential outcomes that would be observed under different treatments \cite {Rubin1974}. Let $X$ be a vector of pre-treatment variables, referred to as covariates, $T$ a binary treatment, with realized value $T=1$ if treated and $T=0$ if control. The causal effect of the treatment is defined as a contrast between two potential outcomes, for example the difference, $Y(1)-Y(0)$, where $Y(1)$ is the potential outcome under treatment and $Y(0)$ is the potential outcome under the control treatment. The observed outcome $Y$ is assumed to be the potential outcomes for each level of the observed treatment $Y=TY(1)+(1-T)Y(0)$, so that the data vector that we observe is $(T_i,X_i,Y_i)$, where $i=1,\ldots,n$ are assumed independent and identically distributed copies. In the remainder of the paper we will drop the subscript i for the random variables when not needed. Since each individual only can be subject to one treatment either $Y(1)$ or $Y(0)$ will be missing. If the treatment is randomized the difference of sample averages of the treated and controls will be an unbiased estimator of the average causal effect $\Delta=E\left[Y(1)-Y(0)\right]$, the parameter of interest. In the following we will use the notation $\mu_1=E\left[Y(1)\right]$ and $\mu_0=E\left[Y(0)\right]$.  When the treatment is not assigned at random the causal effect of the treatment can be estimated if all confounders are observed

\begin{assumption}\label{CIA}[No unmeasured confounding]\\
	$Y(t)\indep T|X, t=0,1$.
\end{assumption}

\noindent and if the treated and controls have overlapping covariate distributions  
\begin{assumption}\label{overlap}[Overlap]\\
	$\eta<P(T=1|X)<1-\eta$, for some $\eta>0$, 
\end{assumption}

\noindent where the assumption that $P(T=1|X)$ is bounded away from zero and one guarantees the existence of a consistent estimator \cite {KT:10}. Throughout the paper we assume that Assumptions \ref{CIA} and \ref{overlap} hold. Under these assumptions we can estimate the  average causal effect with the observed data by marginalizing over the conditional means 

\begin{equation}\label{marg}
\Delta=E\left[E(Y \mid X,T=1)-E(Y \mid X,T=0)\right].
\end{equation}

\noindent For  matching/stratification estimators, see \cite{Imbens2009} for a review, the inner expectation in (\ref{marg}) is evaluated by grouping treated and controls in matched pairs or strata formed by the resulting cells of the cross classification of the covariates. Instead of comparing treated and controls on a high dimensional vector of the covariates it is sufficient to condition on a scalar function of the covariates, $e(X)=P(T=1|X)$, called the propensity score \cite {RR:83}. Instead of conditioning on the propensity score as in (\ref{marg}), the propensity score can be used as a weight   

\begin{equation*}
\Delta =E\left[\frac{TY}{e(X)}-\frac{(1-T)Y}{1-e(X)}\right]=E\left[Y(1)-Y(0)\right],
\end{equation*} 

\noindent where the last equality follows from Assumption 1. 

Usually it is assumed that the propensity score and the outcome regression follow parametric models.
\begin{assumption}\label{PS}[Propensity score model]\\
	The propensity score $e(X)$ follows a model $e(X,\beta)$ parametrized by, $\beta=(\beta_1,\ldots, \beta_p)$ and $\hat{e}(X)$ is the estimated propensity score $e(X,\hat{\beta})$ with a $n^{1/2}$-consistent estimator of $\beta$
\end{assumption}

\begin{assumption}\label{OR}[Outcome regression model]\\
	The conditional expectation, $\mu_t(X)=E(Y(t)|X)$, $t=0,1$ follows a model $\mu_t(X,\alpha_t)$, $t=0,1$ parametrized by $\alpha_t=(\alpha_{t1},\ldots, \alpha_{tq_t})$ and $\hat{\mu}_t(X)$ is the estimated outcome regression $\mu_t(X,\hat{\alpha}_t)$ with a $n^{1/2}$-consistent estimator of $\alpha_t$.
\end{assumption}

Consider the example, $e(X,\beta)=[1+\exp(-X'\beta)]^{-1}$ and $\hat{e}(X)$ are the fitted values of the propensity score when $\hat{\beta}$ is a maximum likelihood estimator of $\beta$. Similarly the outcome regression model could be a linear model $\mu_t(X,\alpha_t)=X'\alpha_t$ where $\hat{\mu}_t(X)$, $t=0,1$ are the fitted values when $\hat{\alpha}_t$ is the ordinary least squares estimator. 

We study two IPW estimators and a DR estimator described in \citeN{LD:04}. We denote by $\hat{\Delta}_{\text{\tiny{IPW}}_1}$ an estimator defined by: 

\begin{equation}\label{f1} 
\hat{\Delta}_{\text{\tiny{IPW}}_1}=\frac{1}{n}\sum_{i=1}^n\frac{T_i{Y_i}}{\hat e(X_i)}-\frac{1}{n}\sum_{i=1}^n\frac{(1-T_i){Y_i}}{1-\hat e(X_i)}.
\end{equation}

\noindent The variance of $\hat{\Delta}_{\textrm{IPW}_1}$ is 

\begin{equation}\sigma_{\text{\tiny{IPW}}_1}^{2}=V_{\text{\tiny IPW}_1}-a^{T}I^{-1}a\end{equation}
\noindent where $V_{{\text{\tiny{IPW}}}_1}$ is the asymptotic variance when the propensity score is known
\begin{equation*}
V_{{\text{\tiny{IPW}}}_1}=  E\left[\frac{Y(1)^2}{e(X)}+\frac{Y(0)^2}{1-e(X)}\right]-(\mu_1-\mu_0)^2,
\end{equation*}
\noindent $a$ is a $(p\times 1)$ vector 
\begin{align*}\label{avector}
a=  E\left[\begin{array}{c}
\left(\dfrac{Y(1)}{e(X)}+\dfrac{Y(0)}{1-e(X)}\right)e'(X)
\end{array}\right]
\end{align*}

\noindent for the partial derivatives $e'(X)=\partial/\partial\beta\left\{e(X,\beta)\right\}$ and $I$ is the $p\times p$ covariance matrix of the estimated propensity score.  In $\hat{\Delta}_{\text{\tiny{IPW}}_1}$ each observed treated individual is weighted by $1/e(X)$ and each control is weighted by $1/\left[1-e(X)\right]$. Since the weights generate the missing potential outcomes for each of the treatment and control groups respectively we want to divide the weighted sum with the number of individuals in the generated sample consisting of the observed and missing potential outcomes which may not be $n$ for a given sample \cite{HIR:03}. This gives an IPW estimator $\hat{\Delta}_{\text{\tiny{IPW}}_2}$ with normalized weights   

\begin{equation}\label{f2} 
\hat{\Delta}_{\text{\tiny{IPW}}_2}=\left(\sum_{i=1}^n\frac{T_i}{\hat e(X_i)}\right)^{-1}\sum_{i=1}^n\frac{T_i{Y_i}}{\hat e(X_i)}-\left(\sum_{i=1}^n\frac{1-T_i}{1-\hat e(X_i)}\right)^{-1}\sum_{i=1}^n\frac{(1-T_i){Y_i}}{1-\hat e(X_i)},
\end{equation}

\noindent and variance

\begin{equation}\sigma_{\text{\tiny{IPW}}_2}^2=V_{\text{\tiny IPW}_2}-b^{T}I^{-1}b \end{equation}
\noindent where $V_{{\text{\tiny{IPW}}}_2}$ is the asymptotic variance when the propensity score is known
\begin{equation*}
V_{{\text{\tiny{IPW}}}_2}=  E\left[\dfrac{\left(Y(1)-\mu_1\right)^{2}}{e(X)}+\dfrac{\left(Y(0)-\mu_0\right)^{2}}{1-e(X)}\right],
\end{equation*}
\noindent and $b$ is a $(p\times 1)$ vector
\begin{align*}
b=  E\left[\begin{array}{c}
\left(\dfrac{Y(1)-\mu_1}{e(X)}+\dfrac{Y(0)-\mu_0}{1-e(X)}\right)e'(X)
\end{array}\right].
\end{align*}

Under Assumptions 1-3 the IPW estimators are consistent estimators of the average causal effect $\Delta$ with asymptotic distribution $\sqrt{n}(\hat{\Delta}_{\text{\tiny{IPW}}_k}-\Delta)\sim N(0,\sigma^2{\text{\tiny{IPW}}_k})$, $k=1,2$.
\vskip 0.5cm

\noindent In addition we study a DR estimator \cite {LD:04,AT:07}
\begin{align} \label{f5}
\nonumber \hat{\Delta}_{\text{\tiny{DR}}}=&\frac{1}{n}\sum_{i=1}^n\frac{T_i{Y_i}-(T_i-\hat e(X_i))\hat{\mu}_{1}(X_i)}{\hat e(X_i)}\\
-&\frac{1}{n}\sum_{i=1}^n\frac{(1-T_i){Y_i}+(T_i-\hat e(X_i))\hat{\mu}_{0}(X_i)}{1-\hat e(X_i)}.
\end{align}

Under Assumptions 1-4 we have the large sample distribution $\sqrt{n}(\hat{\Delta}_{\text{\tiny{DR}}}-\Delta)\sim N(0,\sigma_{\text{\tiny{DR}}}^{2})$ where

\begin{equation}
\sigma_{\text{\tiny{DR}}}^{2}=V_{{\text{\tiny{IPW}}}_2}-d,
\end{equation}
\noindent and
\begin{align*}\label{drb}
d= & E\left[\left(\sqrt{\dfrac{1-e(X)}{e(X)}}\left(\mu_1(X)-\mu_1\right)+\sqrt{\dfrac{e(X)}{1-e(X)}}\left(\mu_0(X)-\mu_0\right)\right)\right]^{2},
\end{align*}

\noindent with the property that $\sigma_{\text{\tiny{DR}}}^{2}\leq \sigma_{\text{\tiny{IPW}}_1}^2,\sigma_{\text{\tiny{IPW}}_2}^2$ which was shown by the theory of Robins and colleagues \cite {RRZ:94}.

\section{Model misspecification: a general approach}\label{general}

Our interest lies in the behaviors of the estimators when the propensity score and the outcome regression models are misspecified. For this purpose we replace Assumptions \ref{PS} and \ref{OR} with two other assumptions defining the probability limit of the estimators under a general misspecification. The misspecifications will further be used to define a general bias of the IPW and DR-estimators.  When the propensity score is misspecified an estimator, e.g., a quasi maximum likelihood estimator (QMLE) is not consistent for $\beta$ in Assumption \ref{PS}. However, a probability limit for an estimator under model misspecification exists under general conditions, see e.g. \citeN[Theorem 2.2]{HW:82} for QMLE or \citeN[Section 12.1]{wooldridge2010econometric} and \citeN[Theorem 7.1]{SB:13} for estimators that can be written as a solution of an estimating equation (M-estimators).

In the following, and as an alternative to Assumptions 3 and 4, we will assume that such limits exists. Below we define an estimator $\hat{e}^*(X)$ of the propensity score under a misspecified model $e^{\text{\tiny{mis}}}(X,\beta^*)$.

\begin{assumption}\label{FPS}[Misspecified PS model parameters] \\ 
	Let $\hat{\beta}^*$ be an estimator under model misspecification, $e^{\text{\tiny{mis}}}(X,\beta^*)$, then  $\hat{\beta}^*\stackrel{p}{\longrightarrow}\beta^*$.  
\end{assumption}

\noindent Under model misspecification the probability limit of $\hat{\beta^*}$ is generally well defined however $e^{\text{\tiny{mis}}}(X,\beta^*)$ is not equal to the propensity score $e(X)$. In the following we use the notation $\hat{e}^*(X)=e^{\text{\tiny{mis}}}(X,\hat{\beta}^*)$ as the estimated propensity score and $e^*(X)=e^{\text{\tiny{mis}}}(X,\beta^*)$ under Assumption \ref{FPS}. Below we give an example for true and misspecified parametric models, however, for Assumption \ref{FPS} we do not need the existence of a true parametric model.  
\vskip 0.5cm

\begin{example}\label{ex}
	For one confounder $X$ and a true PS model $e(X,\beta)=[1+\exp(-\beta_0-\beta_1X-\beta_2X^2)]^{-1}$ assume that we misspecify the propensity score with a probit model $e^{\text{\tiny{mis}}}(X,\beta^*)=\Psi(-\beta^*_0-\beta^*_1X)$, i.e., we misspecify the link function and omit a second order term. Let $\hat{\beta}^*=(\hat{\beta}^*_0,\hat{\beta}^*_1)$ be the QMLE estimator of the parameters in  $e^{\text{\tiny{mis}}}(X,\beta^*)$ obtained by maximizing the quasi-likelihood $$\ln \mathcal{L}=\sum_{i=1}^n\left(T_i\ln e^{\text{\tiny{mis}}}(X_i,\beta^*)+(1-T_i)\ln(1-e^{\text{\tiny{mis}}}(X_i,\beta^*))\right),$$ Then
	$\hat{e}^*(X)=\Psi(-\hat{\beta}^*_0-\hat{\beta}^*_1X)$, $\hat{\beta}^*=(\hat{\beta}^*_0,\hat{\beta}^*_1)\stackrel{p}{\longrightarrow}\beta^*=(\beta^*_0,\beta^*_1)$ under Assumption \ref{FPS} and $e^*(X)=\Psi(-\beta_0^*-\beta_1^*X)$ .  \\
\end{example} 

When considering the existence of true and misspecified parametric models, as illustrated in Example \ref{ex}, the parameters in $\beta$ and the limiting parameters $\beta^*$ under the misspecified model need not to be of the same dimension. For instance, the true model could contain higher order terms and interactions that are not present in the estimation model. 

The next assumption concerns overlap under model misspecification.

\begin{assumption}\label{overlap2}[Overlap under misspecification]\\
	$\nu<e^*(X)<1-\nu$, for some $\nu>0$.
\end{assumption}

In addition to the PS model we also consider misspecified outcome regression models, $\mu^{\text{\tiny{mis}}}_t(X,\alpha^*_t)$, $t=0,1$. Denote by $\hat{\alpha}^*_t$, $t=0,1$ the estimator of the parameters in $\mu^{\text{\tiny{mis}}}_t(X,\alpha^*_t)$. 

\begin{assumption}\label{FOR}[Misspecified OR model parameters]\\
	Let $\hat{\alpha}^*_t$ be an estimator under model misspecification $\mu^{\text{\tiny{mis}}}_t(X,\alpha^*_t)$, $t=0,1$, then $\hat{\alpha}^*_t\stackrel{p}{\longrightarrow}\alpha^*_t$, $t=0,1$.
\end{assumption}

In the following we use the notation $\hat{\mu}^*_t(X)=\mu^{\text{\tiny{mis}}}_t(X,\hat{\alpha}^*_t)$ as the estimated OR and $\mu^*_t(X)=\mu^{\text{\tiny{mis}}}_t(X,\alpha^*_t)$ under Assumption \ref{FOR} and $\mu^*_t$ for the expected value $E\left[\mu^*_t(X)\right]$, $t=0,1$.

Assumptions \ref{FPS} and \ref{FOR} are defined for misspecified PS and OR models for the purpose of describing their influence on the estimation of $\Delta$. The estimators (\ref{f1}), (\ref{f2}) and (\ref{f5}) can be written by estimating equations where the equations solving for the PS and OR parameters are set up below the main equation for the IPW and DR estimators, see e.g \citeN{LD:04} and \citeN{EW2014}. Assuming parametric PS and OR models the IPW estimators correspond to solving $2+p$ estimating equations $\sum_{i=1}^n\psi(\theta,Y_i, T_i, X_i)=0$ for the parameters $\theta_{\text{\tiny{IPW}}_k}=(\mu_1,\mu_0,\beta)$, $k=1,2$ and for the DR estimator $2+p+q_1+q_0$ estimating equations for the parameters $\theta_{\text{\tiny{DR}}}=(\mu_1,\mu_0,\beta,\alpha_1, \alpha_0)$. Using the notation for the misspecified models in Assumptions \ref{FPS} and \ref{FOR} the estimating equations change according to the dimensions of the parameters $\beta^*$ and $\alpha^*_t$, $t=0,1$. A key condition for Assumptions \ref{FPS} and and \ref{FOR} to hold is that the misspecification of the PS and/or OR provides estimating equations that uniquely define the parameter although, as a consequence of the misspecification, it will not be the true average causal effect. In the next section we present the asymptotic bias for the IPW and DR estimators under study with general expressions including the limits of the misspecified propensity score and outcome regression models.

\section{Bias resulting from model misspecification}\label{BIAS}

\subsection{General biases}\label{bias.sec}
In order to study the large sample bias of $\hat{\Delta}_{\text{\tiny{IPW}}_1}$, $\hat{\Delta}_{\text{\tiny{IPW}}_2}$ and $\hat{\Delta}_{\text{\tiny{DR}}}$ under model misspecification we define the estimators  $\hat{\Delta}^*_{\text{\tiny{IPW}}_1}$, $\hat{\Delta}^*_{\text{\tiny{IPW}}_2}$ and $\hat{\Delta}^*_{\text{\tiny{DR}}}$  by replacing $\hat{e}(X)$ in Equations (\ref{f1}), (\ref{f2}) and (\ref{f5}) with $\hat{e}^*(X)$. For the DR-estimator we additionally replace $\hat{\mu}_t(X)$ with $\hat{\mu}^*_t(X)$, $t=0,1$. 

To assess the properties of the estimators we assume 1, 2, 5, 6 and 7 and regularity conditions for applying a weak law of large numbers for averages with estimated parameters, see Appendix \ref{App:AppendixA}. Note that Assumption \ref{PS} and \ref{OR} are no longer needed. We evaluate the difference between the probability limits of the estimators under model misspecification and the average causal effect $\Delta$ for the IPW and DR estimators:

\begin{theorem}[Bias under model misspecification for $\hat{\Delta}_{\text{\tiny{IPW}}_1}^*$] \label{biasIPW1} Under Assumptions 1-2 and 5-6 
	\begin{equation*}
	\hat{\Delta}_{\text{\tiny{IPW}}_1}^*-\Delta\stackrel{p}{\longrightarrow}E\left[\frac{e(X)}{e^*(X)}\mu_1(X)\right]
	-E\left[\frac{1-e(X)}{1-e^*(X)}\mu_0(X)\right]-(\mu_1-\mu_0).
	\end{equation*}
\end{theorem}

\begin{theorem}[Bias under model misspecification for $\hat{\Delta}_{\text{\tiny{IPW}}_2}^*$]\label{biasIPW2} Under Assumptions 1-2 and 5-6 
	\begin{equation*} 
	\hat{\Delta}_{\text{\tiny{IPW}}_2}^*-\Delta\stackrel{p}{\longrightarrow}\\
	\frac{E\left[\frac{e(X)}{e^*(X)}\mu_1(X)\right]}{E\left[\frac {e(X)}{e^*(X)}\right]}
	-\frac{E\left[\frac{1-e(X)}{1-e^*(X)}\mu_0(X)\right]}{E\left[\frac {1-e(X)}{1-e^*(X)}\right]}-(\mu_1-\mu_0).
	\end{equation*}
\end{theorem}

\begin{theorem}[Bias under model misspecification for $\hat{\Delta}_{\text{\tiny{DR}}}^*$] \label{biasDR} Under Assumptions 1-2 and 5-7
	\begin{align*}
	\hat{\Delta}_{\text{\tiny{DR}}}^*-\Delta\stackrel{p}{\longrightarrow}&E\left[\frac{\left(e(X)-e^*(X)\right)\left(\mu_1(X)-{\mu}^*_{1}(X)\right)}{e^*(X)}\right]\nonumber\\
	&+E\left[\frac{\left[e(X)
		-e^*(X)\right]\left(\mu_{0}(X)-\mu^*_0(X)\right)}{\left(1-e^*(X)\right)}\right].
	\end{align*}	
\end{theorem} 	

See Appendix \ref{app1b} for proofs. 
\vskip 0.35cm

We refer to the limits in Theorem \ref{biasIPW1}, \ref{biasIPW2},  and \ref{biasDR} as the asymptotic biases of the respective estimators, i.e., Bias$(\hat{\Delta}_{\text{\tiny{IPW}}_1}^*)$, Bias$(\hat{\Delta}_{\text{\tiny{IPW}}_2}^*)$ and Bias$(\hat{\Delta}_{\text{\tiny{DR}}}^*)$ although they are the difference between the probablity limits of the estimators and the true $\Delta$ and not the difference in expectations. The double robustness property of $\hat{\Delta}_{\text{\tiny{DR}}}^*$ is displayed by Theorem \ref{biasDR} since if either $e(X)=e^*(X)$ or $\mu_t(X)=\mu_t^*(X)$, $t=0,1$ we have that $\hat{\Delta}_{\text{\tiny{DR}}}^*\stackrel{p}{\longrightarrow}\Delta$. 

To provide an illustrative example of the biases of the estimators we obtain the misspecified models' limits by misspecifying the link functions in generalized linear models. However, other data generating processes under Assumptions 1-2, 5-7 could also be used. For the propensity score we use binary response models with logit link (true) and a complementary loglog link (misspecified), for the outcome regression models we use poisson models with log links (true) and gaussian models (misspecified) with identity links. We use numerical approximations to provide values on the parameters in $e^*(X)$ and $\mu^*_t(X)$, $t=0,1$ under the given true and misspecified models $e(X)$, $e^{\text{\tiny{mis}}}(X)$, $\mu_t(X)$ and $\mu^{\text{\tiny{mis}}}_t(X)$, $t=0,1$. \\

\begin{example}\label{ex2}[Bias from link misspecifications]\\
	Let $X\sim\text{Uniform}(-2,2)$ and $T\sim{Bernoulli}(e(X))$. Assume that
	\begin{align*}
	& e(X)=\left[{1+\exp{(0.5-X)}}\right]^{-1}, \quad  e^*(X)=1-\exp\left[-\exp(-0.81+0.74X)\right],\\
	& \mu _1(X)=\exp{(2.3+0.14X)},\quad \mu _0(X)=\exp{(1.4+0.20X)}, \\ 
	&\mu^* _1(X)=10.06+1.48X,\quad \mu^*_0(X)=4.14+0.79X. 
	\end{align*}
	The marginal means are $\mu _1=10.11$ and $\mu _0=4.16$, so that $\Delta=5.94$. Here we have that, Bias$(\hat{\Delta}_{\text{\tiny{IPW}}_1}^*)=-0.16$, Bias$(\hat{\Delta}_{\text{\tiny{IPW}}_2}^*)=0.05$, and Bias$(\hat{\Delta}_{\text{\tiny{DR}}}^*)=-0.02$
\end{example}

\subsection{Comparisons}

To analyze the impact of the model misspecification on the estimators' biases in Section \ref{bias.sec} we compare the biases for two parts separately. The first part concerns the bias with respect to $\mu_1$ and the second part with respect to $\mu_0$.  The first part of Bias$(\hat{\Delta}_{\text{\tiny{IPW}}_1}^*)$ in Theorem \ref{biasIPW1} is  

\begin{align}\label{ipw1}
E\left[\frac{e(X)}{e^*(X)}\mu_1(X)\right]-\mu_1&=cov\left[\frac{e(X)}{e^*(X)},\mu_1(X)\right]+E\left[\frac{e(X)}{e^*(X)}-1\right]\mu_1,
\end{align}

\noindent and the first part of  Bias$(\hat{\Delta}_{\text{\tiny{IPW}}_2}^*)$ in Theorem \ref{biasIPW2} is

\begin{align}\label{ipw2}
\frac{E\left[\frac{e(X)}{e^*(X)}\mu_1(X)\right]}{E\left[\frac{e(X)}{e^*(X)}\right]}-\mu_1= \frac{cov\left[\frac{e(X)}{e^*(X)},\mu_1(X)\right]}{E\left[\frac{e(X)}{e^*(X)}\right]}.
\end{align}   

\noindent For $\hat{\Delta}_{\text{\tiny{IPW}}_1}^*$ we see in (\ref{ipw1}) that the mean difference between the expected value of the conditional outcome, scaled with $E[e(X)/e^*(X)]$, and the marginal outcome contributes to the bias. For $\hat{\Delta}_{\text{\tiny{IPW}}_2}^*$ the contribution (\ref{ipw2}) is the difference between the conditional and the marginal outcome with the same error scaling, but here, the expected value of  $E\left[e(X)/e^*(X)\right]$ also enters the bias in the denominator. For $\hat{\Delta}_{\text{\tiny{IPW}}_1}^*$ we see from the right hand side of (\ref{ipw1}) that the sign depends on the covariance of $e(X)/e^*(X)$ and $\mu_1(X)$, and the sign of the product of $E\left[e(X)/e^*(X)-1\right]$ and $\mu_1$. For $\hat{\Delta}_{\text{\tiny{IPW}}_2}^*$ we see from the right hand side of (\ref{ipw2}) that the sign of depends on the covariance only. Hence, the part of the biases described above can be in different directions for the same model misspecification, see also Example \ref{ex2} for the bias in total. It is no surprise that the covariance of  $e(X)/e^*(X)$ and $\mu_1(X)$  (and similarly of $\left[1-e(X)\right]/\left[1-e^*(X)\right]$  and $\mu_0(X)$) plays a role for the bias of the estimators. If $\mu_1(X)$ was a constant it could be taken out of the expectations of the first terms in (\ref{1}) and (\ref{2}) and the PS-model ratio, $e(X)/e^*(X)$, would be cancelled by the denominator $E\left[e(X)/e^*(X)\right]$. In this case the bias for $\hat{\Delta}_{\text{\tiny{IPW}}_2}^*$ would be 0, and thus smaller than the bias of $\hat{\Delta}_{\text{\tiny{IPW}}_1}^*$ . 

In the sequel we will give results concerning the absolute values of the first part of the biases in Theorems \ref{biasIPW1}, \ref{biasIPW2} and \ref{biasDR} but the results can be directly applied for the second part of the biases by replacing $e(X)/e^*(X)$, with $(1-e(X))/(1-e^*(X))$ and $\mu_1(X)$ with $\mu_0(X)$, see Appendix \ref{app2}. We define $\text{Bias}_1(\hat{\Delta}_{\text{\tiny{IPW}}_1}^*)$, $\text{Bias}_1(\hat{\Delta}_{\text{\tiny{IPW}}_2}^*)$ and $\text{Bias}_1(\hat{\Delta}_{\text{\tiny{DR}}}^*)$ as

\begin{align}
&\text{Bias}_1(\hat{\Delta}_{\text{\tiny{IPW}}_1}^*)=E\left[\frac{e(X)}{e^*(X)}\mu_1(X)\right]
-\mu_1,\label{1}\\
&\text{Bias}_1(\hat{\Delta}_{\text{\tiny{IPW}}_2}^*)=\frac{E\left[\frac{e(X)}{e^*(X)}\mu_1(X)\right]}{E\left[\frac{e(X)}{e^*(X)}\right]}-\mu_1,\label{2}\\
&\text{Bias}_1(\hat{\Delta}_{\text{\tiny{DR}}}^*)=  E\left[\left(\frac{e(X)}{e^*(X)}-1\right)\left(\mu_1(X)-{\mu}^*_{1}(X)\right)\right].\label{3}
\end{align}

We investigate the difference between the bias of the IPW estimators \eqref{1}, \eqref{2} and the bias of the DR estimator \eqref{3}. Hence, we analyze the question of when two wrong models are better than one. We start by comparing Bias$_1(\hat{\Delta}_{\text{\tiny{DR}}}^*)$ and Bias$_1(\hat{\Delta}_{\text{\tiny{IPW}}_1}^*)$. In the following theorem we show a necessary condition for Bias$_1(\hat{\Delta}_{\text{\tiny{DR}}}^*)$ to be smaller than Bias$_1(\hat{\Delta}_{\text{\tiny{IPW}}_1}^*)$. In the sequel all proofs are provided in Appendix \ref{app2}.

\begin{theorem}[Necessary condition for Bias$_1(\hat{\Delta}_{\text{\tiny{DR}}}^*)$ smaller than Bias$_1(\hat{\Delta}_{\text{\tiny{IPW}}_1}^*)$]\label{dr2} 
	If
	\begin{equation*}
	\left| E\left[\left(\frac{e(X)}{e^*(X)}-1\right)\left(\mu_1(X)-{\mu}^*_{1}(X)\right)\right]\right|<  	\left| E\left[\left(\frac{e(X)}{e^*(X)}-1\right)\mu_1(X)\right]\right|,
	\end{equation*}
	then
	\begin{equation*}
	\left|E\left[\left(\frac{e(X)}{e^*(X)}-1\right){\mu}^*_{1}(X)\right]\right|<  2\cdot\left| E\left[\left(\frac{e(X)}{e^*(X)}-1\right)\mu_1(X)\right] \right|.
	\end{equation*}
	
\end{theorem}

The theorem states that if the DR estimator improves upon the simple IPW-estimator under misspecification of both the PS and the OR model, then, the absolute value of the misspecified outcome model is less than double the absolute value of the true conditional mean under the same scaling of the PS-model error, $e(X)/e^*(X)-1$.  
\vskip 0.5cm

\newcounter{revcounter}
\newenvironment{rev}{%
	\refstepcounter{revcounter}%
	\textbf{Example 2.\arabic{revcounter}}%
	\quad
}

\begin{rev}[Numerical example 1 revisited for Bias$_1(\hat{\Delta}_{\text{\tiny{DR}}}^*)$ and Bias$_1(\hat{\Delta}_{\text{\tiny{IPW}}_1}^*)$]\\
	For the data generating process in Example 2 we investigate the first part of the bias
	
	\begin{align*}
	\left| E\left[\left(\frac{e(X)}{e^*(X)}-1\right)\left(\mu_1(X)-{\mu}^*_{1}(X)\right)\right]\right|&<  	
	\left| E\left[\left(\frac{e(X)}{e^*(X)}-1\right){\mu}_{1}(X)\right]\right|\\
	0.01<0.11
	\end{align*}
	
	implies
	\begin{align*}
	\left|E\left[\left(\frac{e(X)}{e^*(X)}-1\right){\mu}_{1}^*(X)\right]\right|&< 2\cdot\left| E\left[\left(\frac{e(X)}{e^*(X)}-1\right){\mu}_{1}(X)\right] \right|\\
	0.10&<0.22
	\end{align*}

	\noindent which is consistent with Theorem \ref{dr2}. 
	
\end{rev}

Below we give two examples of sufficient conditions for the DR-estimator to have a smaller bias than the simple IPW estimator.

\begin{theorem}[Sufficient conditions for Bias$(\hat{\Delta}_{\text{\tiny{DR}}}^*)$ smaller than Bias$(\hat{\Delta}_{\text{\tiny{IPW}}_1}^*)$]\label{dr3} 
	If 
	a) $\mu_1^*=\mu_1$ and $0<E\left[\frac{e(X)}{e^*(X)}-1\right] \mu_1<cov\left[\frac{e(X)}{e^*(X)},\mu_1^*(X)\right]<cov\left[\frac{e(X)}{e^*(X)},\mu_1(X)\right]$\\
	or, 
	
	b)    	 \begin{equation*}
	\left|E\left[\left(\frac{e(X)}{e^*(X)}-1\right){\mu}^*_{1}(X)\right]\right|<  2\cdot\left| E\left[\left(\frac{e(X)}{e^*(X)}-1\right)\mu_1(X)\right] \right|,
	\end{equation*}
	and $E\left[\left(\frac{e(X)}{e^*(X)}-1\right){\mu}^*_{1}(X)\right]$ and $E\left[\left(\frac{e(X)}{e^*(X)}-1\right){\mu}_{1}(X)\right]$ are either both positive or both negative, then
	
	\begin{equation*}
	\left| E\left[\left(\frac{e(X)}{e^*(X)}-1\right)\left(\mu_1(X)-{\mu}^*_{1}(X)\right)\right]\right|<  	\left| E\left[\left(\frac{e(X)}{e^*(X)}-1\right)\mu_1(X)\right]\right|.
	\end{equation*}
	
\end{theorem}

One of the criteria in a), that $\mu_1=\mu_1^*$, is reasonable to assume when the corresponding moment condition is used in the estimation of the misspecified outcome model. Also, we have that criterion b) is the same as the necessary condition with the added assumption that the expectation of the (PS-error scaled) conditional outcomes have the same sign.

\begin{theorem}[Necessary condition for Bias$_1(\hat{\Delta}_{\text{\tiny{DR}}}^*)$ smaller than Bias$_1(\hat{\Delta}_{\text{\tiny{IPW}}_2}^*)$]\label{dr4} 
	If
	\begin{equation*}
	\left| E\left[\left(\frac{e(X)}{e^*(X)}-1\right)\left(\mu_1(X)-{\mu}^*_{1}(X)\right)\right]\right|<  \left|\frac{E\left[\frac{e(X)}{e^*(X)}\mu_1(X)\right]}{E\left[\frac {e(X)}{e^*(X)}\right]}-\mu_1\right|,
	\end{equation*}
	then
	\begin{align*}
	&E\left[\left(\frac{e(X)}{e^*(X)}-1\right)\left(\mu_1(X)\right)\right]-\frac{\left|cov\left[\frac{e(X)}{e^*(X)}, \mu_1(X)\right]\right|}{E\left[\frac {e(X)}{e^*(X)}\right]}< E\left[\left(\frac{e(X)}{e^*(X)}-1\right)\left[{\mu}^*_{1}(X)\right]\right]\\
	<&E\left(\left[\frac{e(X)}{e^*(X)}-1\right]\left[\mu_1(X)\right]\right)+\frac{\left|cov\left[\frac{e(X)}{e^*(X)}, \mu_1(X)\right]\right|}{E\left[\frac {e(X)}{e^*(X)}\right]}.
	\end{align*}
	
\end{theorem}

From the theorem we see that for the DR estimator to improve upon the normalized IPW estimator we need that the outcome misspecification is within an interval defined by the true conditional outcome and the absolute value of the covariance. This means that the smaller the covariance is, the more accuracy of the outcome model is required for the $\hat{\Delta}_{\text{\tiny{DR}}}^*$ to be less biased than $\hat{\Delta}_{\text{\tiny{IPW}}_2}^*$.
\vskip 0.5cm

\begin{rev}[Example 2 revisited for Bias$_1(\hat{\Delta}_{\text{\tiny{DR}}}^*)$ and Bias$_1(\hat{\Delta}_{\text{\tiny{IPW}}_2}^*)$]\\
	For the data generating process in Example 2 we investigate the first part of the bias
	
	\begin{align*}
	\left| E\left[\left(\frac{e(X)}{e^*(X)}-1\right)\left(\mu_1(X)-{\mu}^*_{1}(X)\right)\right]\right|&<  	
	\left|\frac{E\left[\frac{e(X)}{e^*(X)}\mu_1(X)\right]}{E\left[\frac {e(X)}{e^*(X)}\right]}-\mu_1\right|\\
	0.01<0.06
	\end{align*}
	implies
	\begin{align*}
	E\left[\left(\frac{e(X)}{e^*(X)}-1\right){\mu}_{1}^*(X)\right]&> E\left[\left(\frac{e(X)}{e^*(X)}-1\right)\left(\mu_1(X)\right)\right]-\frac{\left|cov\left[\frac{e(X)}{e^*(X)}, \mu_1(X)\right]\right|}{E\left[\frac {e(X)}{e^*(X)}\right]}\\
	-0.10&>-0.11-0.06=-0.17
	\end{align*}
	and 
	\begin{align*}
	E\left[\left(\frac{e(X)}{e^*(X)}-1\right){\mu}_{1}^*(X)\right]&<E\left[\left(\frac{e(X)}{e^*(X)}-1\right)\left(\mu_1(X)\right)\right]+\frac{\left|cov\left[\frac{e(X)}{e^*(X)}, \mu_1(X)\right]\right|}{E\left[\frac {e(X)}{e^*(X)}\right]}\\
	-0.10&<-0.11+0.06=-0.05
	\end{align*}
	
	\noindent which is consistent with Theorem \ref{dr4}. 
	
\end{rev}

In Theorem \ref{dr5} we give examples of sufficient conditions for the comparison of Bias$_1(\hat{\Delta}_{\text{\tiny{DR}}}^*)$ and Bias$_1(\hat{\Delta}_{\text{\tiny{IPW}}_2}^*)$.

\begin{theorem}[Sufficient conditions for Bias$_1(\hat{\Delta}_{\text{\tiny{DR}}}^*)$ smaller than Bias$_1(\hat{\Delta}_{\text{\tiny{IPW}}_2}^*)$]\label{dr5} 
	
	If
	a)
	$\mu_1=\mu_1^*$ and\\ $cov\left[\frac{e(X)}{e^*(X)},\mu_1(X)\right]-\left|\frac{cov\left[\frac{e(X)}{e^*(X)},\mu_1(X)\right]}{E\left[\frac {e(X)}{e^*(X)}\right]}\right|<cov\left[\frac{e(X)}{e^*(X)}, \mu_1^*(X)\right]
	<cov\left[\frac{e(X)}{e^*(X)}, \mu_1(X)\right]+\left|\frac{cov\left[\frac{e(X)}{e^*(X)},\mu_1(X)\right]}{E\left[\frac {e(X)}{e^*(X)}\right]}\right|$ \\
	or if,\\
	b)   $E\left[\left(\frac{e(X)}{e^*(X)}-1\right){\mu}_{1}(X)\right]$ and $cov\left[\frac{e(X)}{e^*(X)}, \mu_1(X)\right]$ are either both positive or both negative, 
	
	and
	\begin{equation*}
	\left|E\left[\left(\frac{e(X)}{e^*(X)}-1\right){\mu}_{1}^*(X)\right]\right|<\left|E\left[\left(\frac{e(X)}{e^*(X)}-1\right){\mu}_{1}(X)\right]+\frac{cov\left[\frac{e(X)}{e^*(X)}, \mu_1(X)\right]}{E\left[\frac {e(X)}{e^*(X)}\right]}\right|
	\end{equation*}
	
	then,
	
	\begin{equation*}
	\left| E\left[\left(\frac{e(X)}{e^*(X)}-1\right)\left(\mu_1(X)-{\mu}^*_{1}(X)\right)\right]\right|<  \left|\frac{E\left[\frac{e(X)}{e^*(X)}\mu_1(X)\right]}{E\left[\frac {e(X)}{e^*(X)}\right]}-\mu_1\right|.
	\end{equation*}
\end{theorem}

Summarizing the results from the comparisons of Theorems \ref{dr2}, \ref{dr3}, \ref{dr4} and \ref{dr5} we note that the expected value of the product of the PS-model error and the true and misspecified conditional outcomes play important roles. Here, the covariances of the PS-model ratio and the true and misspecified conditional outcomes are two of their respective components. In Figure 1 we illustrate these parts with the data generating processes from Example \ref{ex2}. The PS-model ratio deviates from 1 for both small and large values of $X$, however more for smaller values of $X$. Since both conditional outcomes $\mu_1(X)$ and $\mu_1^*(X)$ are strictly increasing both covariances are positive although quite small due to the PS-model error increase for larger values of $X$ ($cov\left[e(X)/e^*(X), \mu_1(X)\right]=0.064$ and $cov\left[e(X)/e^*(X), \mu^*_1(X)\right]=0.074$). The interval characterization of the described conditions implies that if the two covariances are of the same magnitude the bias of $\hat{\Delta}_{\text{\tiny{DR}}}^*$ will often be smaller than the biases of $\hat{\Delta}_{\text{\tiny{IPW}}_1}^*$ and $\hat{\Delta}_{\text{\tiny{IPW}}_2}^*$.  

\begin{figure}[ht!]
	\begin{center}	\caption{Illustration of the components of the biases from Numerical example \ref{ex2}. Top left: $e(X)$ and $e^*(X)$ by $X$, top right: $e(X)/e^*(X)$ by $X$, bottom left: $\mu_1(X)$ and $\mu_1^*(X)$ by $X$ and bottom right: $\mu_1(X)$ and $\mu_1^*(X)$ by $e(X)/e^*(X)$.}
		\includegraphics[width=1\textwidth]{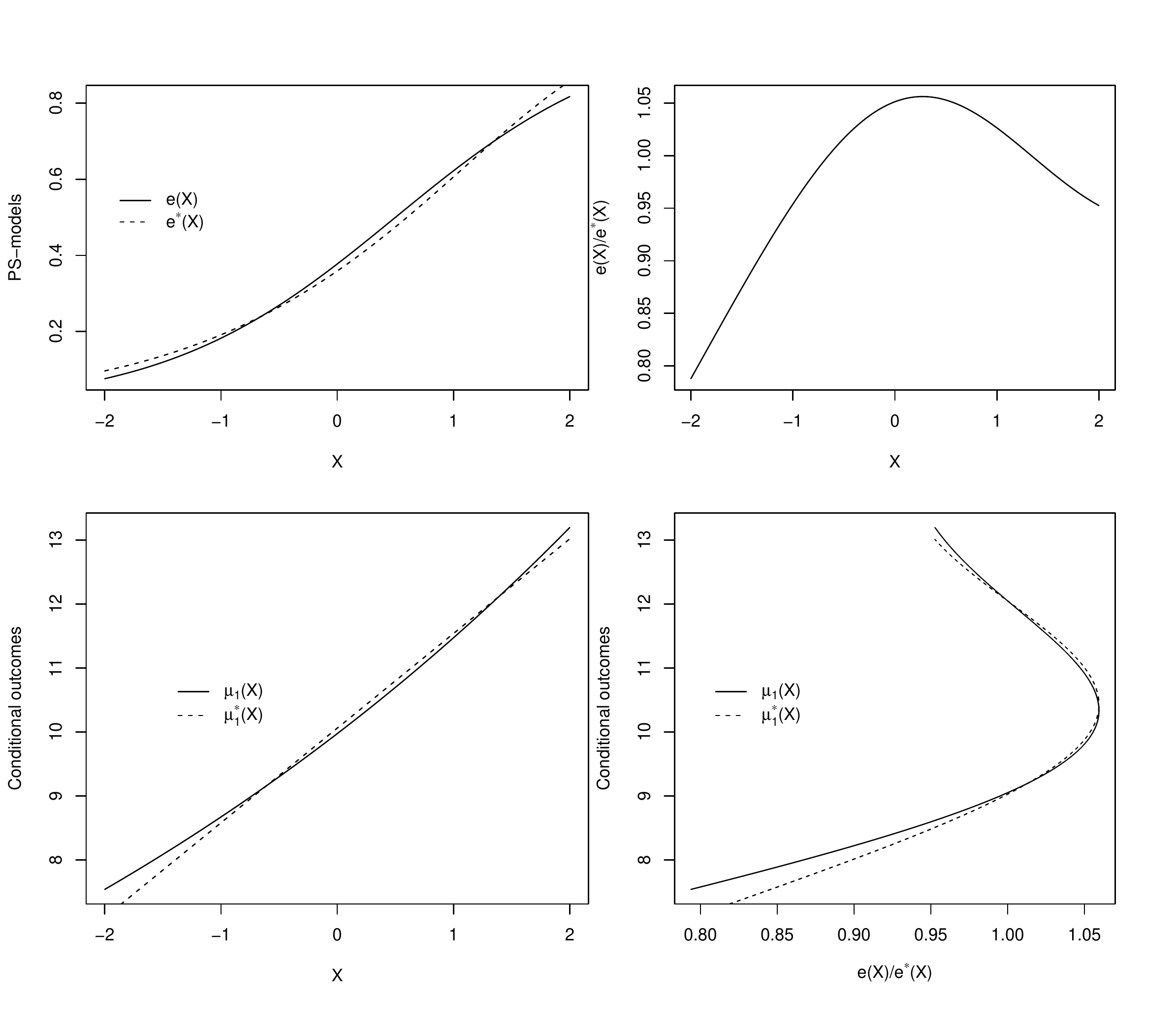}
	\end{center}
\end{figure}

\section{Simulation study}\label{simulations}

In order to investigate the asymptotic biases described in Section \ref{BIAS} and also the finite sample performance of $\hat{\Delta}^*_{\text{\tiny{IPW}}_1}$, $\hat{\Delta}^*_{\text{\tiny{IPW}}_2}$ and $\hat{\Delta}^*_{\text{\tiny{DR}}}$ under model misspecification we perform a simulation study with three different designs A, B and C. The first part of the simulations evaluate the finite sample performances of the estimators and consist of 1000 replications of sample sizes 500, 1000 and 5000.  We generate covariates $X_1\sim$ \textrm{Uniform}(1,4), $X_2\sim$ \textrm{Poisson}(3) and $X_3\sim$ \textrm{Bernoulli}(0.4). We use generalized linear models to generate a binary treatment $T$ and potential outcomes $Y(t)$, $t=0,1$ with second order terms of $X_1$ and $X_2$ in both the PS and OR models. The PS-distributions for the treated and controls are bounded away from zero and 1 under the true models and under the model misspecifications. The PS and OR models (for the DR estimator) are stepwise misspecified. We have three designs where:
\begin{description}
	\item[A] a quadratic term ${X_1}^2$ is omitted in the PS and OR models;
	\item [B] two quadratic terms, ${X_1}^2$ and ${X_2}^2$,  are omitted in the PS and OR models;
	\item [C] two quadratic terms are omitted and the both the OR and PS link functions are misspecified.
\end{description}

The glm family and link functions together with the true parameter values are given in Table \ref{designs} which also contains the details for the misspecified models. The simulation is performed with the statistical software R \cite{R}. 
\vskip 0.5cm
\begin{sidewaystable}[bt]
	\vskip 12cm
	\begin{small}
		\caption{Simulation Designs A, B and C}
		\begin{tabular}{lcccc}
			\toprule
			&	\multicolumn{2}{c}{TRUE MODEL}& \multicolumn{2}{c}{MISSPECIFIED MODEL} \\ 
			Models & Class  & Linear predictor and parameter values& Class & Linear predictor \\
			\midrule
			& & $\beta=(-1,0.6,0.1,0.9,0.1,0.7), \alpha_0=(3,0.5,0.2,0.5,0.2,0.2)$ & \\
			& & $\alpha_1=(4,1.1,0.1,0.5,0.3,0.2)$ & \\
			{\it Design A} & & & &\\
			PS & Binomial, logit & $X_1,X_2,X_1^2,X_2^2,X_3$ & Binomial, logit & $X_1,X_2,X_2^2,X_3$\\
			OR & Gaussian, identity &$X_1,X_2,X_1^2,X_2^2,X_3$& Gaussian, identity&$X_1,X_2,X_2^2,X_3$ \\
			\midrule
			{\it Design B} & & & &\\
			PS & Binomial, logit & $X_1,X_2,X_1^2,X_2^2,X_3$   & Binomial, logit & $X_1,X_2,X_3$\\
			OR & Gaussian, identity & $X_1,X_2,X_1^2,X_2^2,X_3$ & Gaussian, identity&$X_1,X_2,X_3$\\
			\midrule
			{\it Design C} & & & &\\
			PS & Binomial, cauchit & $X_1,X_2,X_1^2,X_2^2,X_3$  &Binomial, logit &$X_1,X_2,X_3$\\
			OR & Gamma, identity & $X_1,X_2,X_1^2,X_2^2,X_3$ & Gaussian, identity& $X_1,X_2,X_3$\\
		\end{tabular}
		\label{designs}
	\end{small}
\end{sidewaystable}

In Table \ref{table1} we give the simulation bias, standard error and MSE of the three estimators. When using the true models, i.e, when studying the estimators  $\hat{\Delta}_{\text{\tiny{IPW}}_1}$, $\hat{\Delta}_{\text{\tiny{IPW}}_2}$ and $\hat{\Delta}_{\text{\tiny{DR}}}$ the bias is small and decreases when the sample size increases and the standard errors follow the expected order with the smallest for $\hat{\Delta}_{\text{\tiny{DR}}}$ followed by $\hat{\Delta}_{\text{\tiny{IPW}}_2}$ and $\hat{\Delta}_{\text{\tiny{IPW}}_1}$ \cite{LD:04}. Under misspecification the bias does not decrease with the sample size but gets closer to the asymptotic biases, see Table \ref{theo}. Under misspecification the standard errors follow the same pattern as under the true models. The bias of $\hat{\Delta}^*_{\text{\tiny{IPW}}_1}$ is the largest whereas the $\hat{\Delta}^*_{\text{\tiny{IPW}}_2}$ and $\hat{\Delta}^*_{\text{\tiny{DR}}}$ are similar. The MSE of $\hat{\Delta}^*_{\text{\tiny{DR}}}$ is the smallest, however for $n=5000$, $\hat{\Delta}^*_{\text{\tiny{IPW}}_2}$ and $\hat{\Delta}^*_{\text{\tiny{DR}}}$ are very similiar.

In Table \ref{theo} we give numerical approximations for Bias$(\hat{\Delta}_{\text{\tiny{IPW}}_1}^*)$,
Bias$(\hat{\Delta}_{\text{\tiny{IPW}}_2}^*)$ and Bias$(\hat{\Delta}_{\text{\tiny{DR}}}^*)$ using a sample size of $n=1,000,000$. We also show the same approximations for Bias$_t(\hat{\Delta}_{\text{\tiny{IPW}}_1}^*)$,
Bias$_t(\hat{\Delta}_{\text{\tiny{IPW}}_2}^*)$ and Bias$_t(\hat{\Delta}_{\text{\tiny{DR}}}^*)$, $t=0,1$. Here, we see that the total bias is smallest for $\hat{\Delta}^*_{\text{\tiny{IPW}}_2}$ in Design A but smaller for $\hat{\Delta}^*_{\text{\tiny{DR}}}$ in Design B and C. The absolute value of the biases in the two parts are smallest for $\hat{\Delta}^*_{\text{\tiny{DR}}}$ in all designs.  We also give the expectations and covariances that are used for the necessary and sufficient conditions in Theorems \ref{dr2}-\ref{dr5}. We immediately see that the necessary condition for the absolute values of Bias$_t(\hat{\Delta}_{\text{\tiny{DR}}}^*)$ to be smaller than the absolute value of Bias$_t(\hat{\Delta}_{\text{\tiny{IPW}}_1}^*)$ holds for both $t=0,1$. 
The means $(\mu_0,\mu_1)$ are close to the means under model misspecification $(\mu^*_0,\mu^*_1)$ which is an assumption needed in order to evaluate the sufficient conditions in Theorems \ref{dr3} a) and \ref{dr5} a). By inspecting the covariances and the additional critera of the Theorems \ref{dr3} a) and \ref{dr5} a) in Table \ref{theo} we can see that the resulting inequalities of the theorems are in line with the results for the asymptotic biases.

Since we have that  $E\left[\left(e(X)/e^*(X)-1\right){\mu}_{1}(X)\right]$ and $cov\left[e(X)/e^*(X), \mu_1(X)\right]$ and\\ $E\left[\left((1-e(X))/(1-e^*(X))-1\right){\mu}_{0}(X)\right]$ and $cov\left[(1-e(X))(1-e^*(X)), \mu_0(X)\right]$ do not have the same sign in Design A and C the sufficient conditions in Theorems \ref{dr3} b) and \ref{dr5} b) cannot be applied.

\begin{figure}[bt]
	\begin{center}	\caption{Density plots of the propensity score distributions, $\hat{e}(X)$ and $\hat{e}^*(X)$ for treated and controls for Design A (top), B (middle), and C (bottom).}
		\includegraphics[width=1\textwidth]{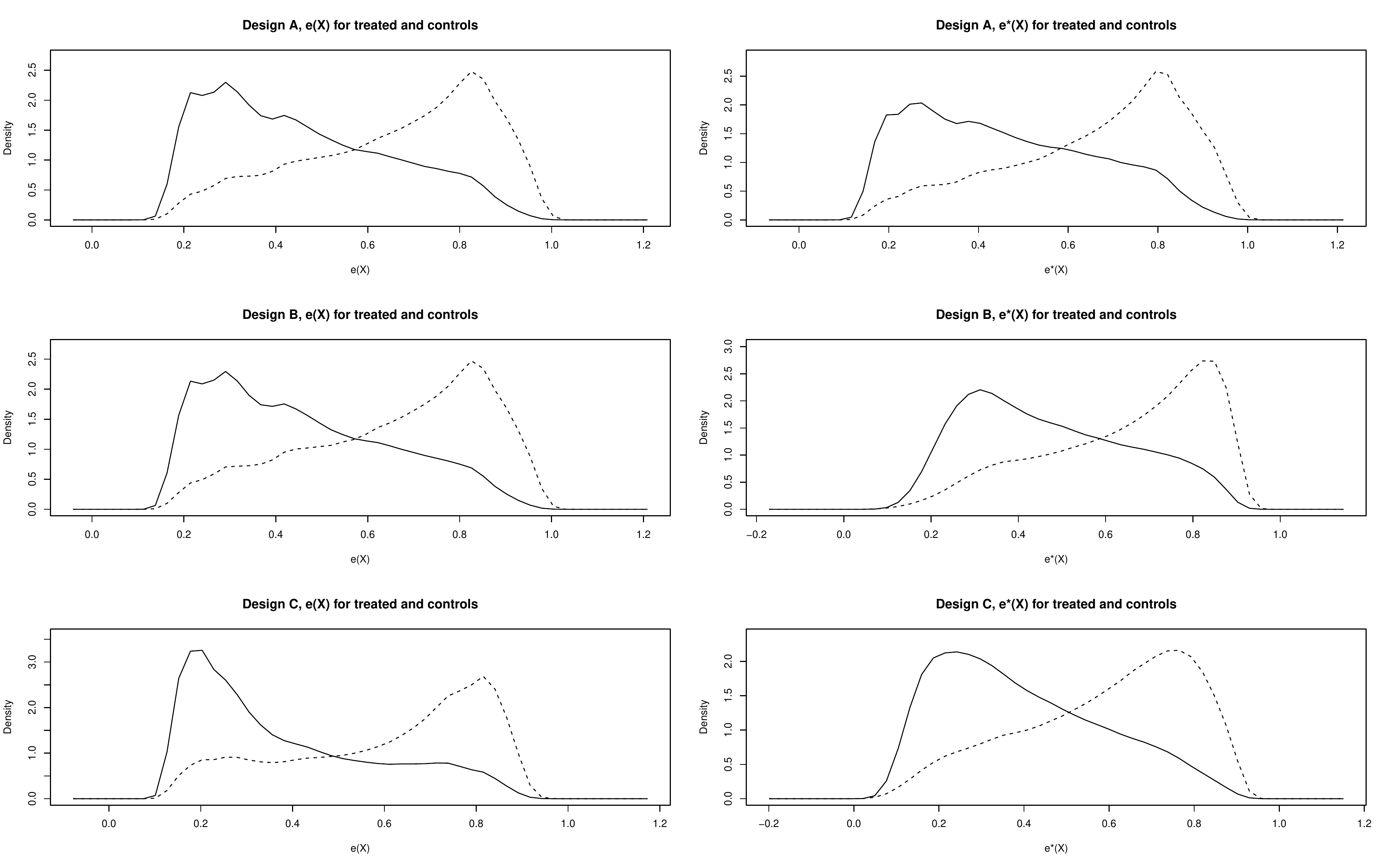}
	\end{center}
\end{figure}\label{biasFig}

\begin{table}[bt]
	\begin{small}
		\caption{Simulation results}
		\begin{tabular}{crc|rrr|rrr|rrr}
			\toprule
			\multicolumn{3}{c}{}&\multicolumn{9}{c}{ESTIMATORS}\\
			&&& \multicolumn{3}{c}{$\hat{\Delta}^*_{\text{\tiny{IPW}}_1}$} & \multicolumn{3}{c}{$\hat{\Delta}^*_{\text{\tiny{IPW}}_2}$}& \multicolumn{3}{c}{$\hat{\Delta}_{\text{\tiny{DR}}}^*$}\\

			\textbf{n} &    Models  &  Design & Bias & SD & MSE & Bias & SD & MSE& Bias & SD & MSE \\
			\midrule
			500 & \multicolumn{1}{c}{True} &  &0.018 & 0.368 & 0.136 &  0.013&0.138&0.019 &$<$0.001 &0.11&0.012 \\
			\textbf{} & \multicolumn{1}{c}{False} & A &0.118&0.390& 0.166&0.019&0.136 &0.019 &0.021&0.119 &0.015\\
			\textbf{} & \multicolumn{1}{c}{False} & B & 0.262&0.324&0.174& 0.033&0.126& 0.017& 0.024&0.110&0.013\\
			\textbf{} & \multicolumn{1}{c}{False} &C &0.229&0.349&0.174& -0.061& 0.128& 0.020&0.037&0.115&0.015 \\
			\midrule
			1000 & \multicolumn{1}{c}{True} & &0.010 & 0.254& 0.065&0.002& 0.098& 0.010&$<0.001$& 0.075& 0.006\\
			\textbf{} & \multicolumn{1}{c}{False} & A &0.112&0.270&0.085&0.008 &0.097&0.009&0.018&0.079&0.007\\
			\textbf{} & \multicolumn{1}{c}{False} & B &0.260&0.214&0.113&0.034&0.088&0.009&0.025&0.078&0.007 \\
			\textbf{} & \multicolumn{1}{c}{False} & C &0.243&0.222&0.108&-0.056&0.092&0.012&0.043&0.081&0.008\\
			\midrule
			5000 & \multicolumn{1}{c}{True}& &0.007&0.107&0.011&0.005&0.044&0.002&0.003&0.035&0.001 \\
			\textbf{} & \multicolumn{1}{c}{False} & A &0.112&0.116&0.026&0.013&0.044&0.002&0.025&0.037&0.002\\
			\textbf{} & \multicolumn{1}{c}{False} & B &0.260&0.090&0.076& 0.033&0.039&0.003& 0.025&0.036&0.002\\
			\textbf{} & \multicolumn{1}{c}{False} & C &0.226&0.096&0.060&-0.057&0.040&0.005&0.040&0.036&0.003\\
			
			\bottomrule
		\end{tabular}
		\label{table1}
	\end{small}
\end{table}

\begin{table}[bt] 
	\caption{Asymptotic approximations from Designs A, B and C.
	}
	\centering
	\begin{tabular}{l|ccc}
		\toprule
		\multicolumn{1}{c}{} & \multicolumn{3}{c}{Design}   \\ 
		Parameter& A & B  & C   \\  
		\midrule
		$\mu_1$ &11.127&11.127&12.130\\
		$\mu_1^*$&11.092&11.098&12.097\\
		$\mu_0$ &8.628&8.628&9.633\\
		$\mu_0^*$&8.578&8.564&9.582\\
		\midrule
		Bias$(\hat{\Delta}_{\text{\tiny{IPW}}_1}^*)$ &0.096&0.264&0.213\\
		Bias$(\hat{\Delta}_{\text{\tiny{IPW}}_2}^*)$ & 0.007&0.033&-0.057\\
		Bias$(\hat{\Delta}_{\text{\tiny{DR}}}^*)$&  0.017&0.025&0.037\\
		\midrule
		Bias$_1(\hat{\Delta}_{\text{\tiny{IPW}}_1}^*)$ & 0.024&0.128&0.130\\
		Bias$_1(\hat{\Delta}_{\text{\tiny{IPW}}_2}^*)$ & -0.028&0.013 &-0.089\\
		Bias$_1(\hat{\Delta}_{\text{\tiny{DR}}}^*)$&  0.009 &0.009&0.029\\
		\midrule
		$E\left[\frac{e(X)}{e^*(X)}\right]$	 & 1.005 & 1.010& 1.019 \\
		$cov\left[\frac{e(X)}{e^*(X)}, \mu_1(X)\right]$ &-0.029& 0.015&-0.095  \\
		$cov\left[\frac{e(X)}{e^*(X)}, \mu_1^*(X)\right]$& -0.040& 0.006& -0.121\\
		\midrule
		$E\left[\frac{e(X)}{e^*(X)}-1\right]\mu_1$ & 0.060&0.113&0.224\\
		$E\left[\left(\frac{e(X)}{e^*(X)}-1\right)\mu_1(X)\right]$ & 0.030&0.128&0.130\\
		$E\left[\left(\frac{e(X)}{e^*(X)}-1\right)\mu_1^*(X)\right]$& 0.019&0.119&0.102\\
		\midrule
		Bias$_2(\hat{\Delta}_{\text{\tiny{IPW}}_1}^*)$&0.076 &0.137&0.083\\
		Bias$_2(\hat{\Delta}_{\text{\tiny{IPW}}_2}^*)$ &0.039&0.022&0.031  \\
		Bias$_2(\hat{\Delta}_{\text{\tiny{DR}}}^*)$& 0.011&0.018&0.008 \\
		\midrule
		
		$E\left[\frac{1-e(X)}{1-e^*(X)}\right]$	 & 0.995 &0.987&0.995 \\
		$cov\left[\frac{1-e(X)}{1-e^*(X)}, \mu_0(X)\right]$ &-0.037&-0.017&-0.033 \\
		$cov\left[\frac{1-e(X)}{1-e^*(X)}, \mu_0^*(X)\right]$&-0.026&-0.004&-0.024 \\
		\midrule
		$E\left[\frac{1-e(X)}{1-e^*(X)}-1\right]\mu_0$ &-0.038 &0.111 &-0.052\\
		$E\left[\left(\frac{1-e(X)}{1-e^*(X)}-1\right)\mu_0(X)\right]$ & -0.074&-0.128&-0.085 \\
		$E\left[\left(\frac{1-e(X)}{1-e^*(X)}-1\right)\mu_0^*(X)\right]$&-0.065&-0.114&-0.075\\
		\bottomrule
	\end{tabular}
	\label{theo}
\end{table}

\section{Discussion}
In this paper we investigate biases of two IPW estimators and a DR estimator under model misspecification. For this purpose, we use a generic probability limit, under  misspecification of the PS and OR models, which exists under general conditions. Since the propensity score enters the estimator in different ways for the IPW estimators under study the consequences of the model misspecification are not the same. The bias of the IPW estimators depend on the covariance between the PS-model error and the conditional outcome in different ways and the resulting bias can be in opposite directions. Comparing the bias of the DR estimator with a simple IPW estimator the necessary condition for the DR estimator to have a smaller bias is that the expectation of the outcome model under misspecification is less than twice the true conditional outcome, where the expectations includes a scaling with the PS-model error. For the comparison with the normalized IPW estimator the (PS-error scaled) misspecified outcome involves an interval defined by the true conditional outcome adding and subtracting the absolute value of the covariance between the PS-model error and the conditional outcome.  

The comparisons of the IPW and DR estimators suggests that in general the DR estimator may have a smaller bias than the IPW estimators, under misspecification of the outcome model, however there is no guarantee that this is the case. Since all biases include the PS-model error we suggest that a researcher should be careful when modelling the PS whenever such a model is involved.

To our knowledge, there are only simulation studies comparing DR-estimators with other estimators \cite {KS:07,IW:12} under the assumption that all models are misspecified. In this paper we study the same problem with an analytical approach although the comparisons are made between a DR estimator and IPW estimators.

\section*{Acknowledgement}
The authors acknowledge the Royal Swedish Academy of Sciences and the Swedish Research Council, (grant number 2016-00703) for financial support. 

\addresseshere 

\bibliographystyle{chicago}

\bibliography{Referenser}

\newpage
\appendix

\section{Appendix}
\subsection{Regularity conditions for applying a weak law of large numbers for averages of functions with estimated parameters}\label{App:AppendixA}

\noindent The convergence in probability of $\hat{\Delta}^*_{\text{\tiny{IPW}}_1}$, $\hat{\Delta}^*_{\text{\tiny{IPW}}_2}$ and $\hat{\Delta}^*_{\text{\tiny{DR}}}$ to their corresponding expectations would follow directly from a WLLN for an iid sample of $(T_i,X_i,Y_i)$ except for the estimated parameters $\hat{\beta}^*$ in $\hat{e}^*(X_i)$ and  and $\hat{\alpha}_t^*$ in $\hat{\mu}^*_t(X_i)$, $t=0,1$. To justify the biases in Section \ref{BIAS} consider a general representation of a function $g\left[T,Y,X,\hat{\theta}\right]$ where $\hat{\theta}\stackrel{p}{\longrightarrow}{\theta}_0$ and 

\begin{equation}\label{WLLN}
\frac{1}{n}\sum g\left[X_i,T_i,Y_i,\hat{\theta}\right]\stackrel{p}{\longrightarrow}E\left[g(T,X,Y,\theta_0)\right]
\end{equation}

\noindent The $\hat{\theta}$ in (\ref{WLLN}) corresponds to $\hat{\beta}^*$ for $\hat{\Delta}^*_{\text{\tiny{IPW}}_1}$, $\hat{\Delta}^*_{\text{\tiny{IPW}}_2}$ and $(\hat{\alpha}^*,\hat{\beta}^*)$ for  $\hat{\Delta}^*_{\text{\tiny{DR}}}$ and under Assumptions \ref{FPS} and \ref{FOR} the consistency of $\hat{\theta}$ is ensured. Regularity conditions for the function $g$ can be given see e.g., citet[Theorem 7.3]{SB:13} who show that (\ref{WLLN}) holds for differentiable functions with bounded derivatives (wrt $\theta$). The regularity conditions for $g\left[X_i,T_i,Y_i,\hat{\theta}\right]$, for the three estimators, imply conditions on the models $e^*(X,\beta^*)$ and $\mu^*_t(X,\alpha_t^*)$ such that the regularity condition for $g$ is satisfied. Under (\ref{WLLN}) we can insert the limiting values $\beta^*$ and $\alpha^*$  and their corresponding $e^*(X)$ and $\mu^*_t(X)$, $t=0,1$ when taking a WLLN. 

\subsection{Biases}\label{app1b}

Under the regularity conditions and Assumptions 5-6 (IPW) and 5-7 (DR) we derive the bias in Theorem \ref{biasIPW1} below. For $\hat{\Delta}^*_{\text{\tiny{IPW}}_1}$: 

\begin{equation*}
\frac{1}{n}\sum_{i=1}^n\frac{T_i{Y_i}}{\hat{e}^*(X_i)}-\frac{1}{n}\sum_{i=1}^n\frac{(1-T_i){Y_i}}{1-\hat {e}^*(X_i)}\stackrel{p}{\longrightarrow}E\left[\frac{TY}{e^*(X)}\right]-E\left[\frac{\left(1-T\right)Y}{\left(1-e^*(X)\right)}\right], 
\end{equation*}

\noindent and 

\begin{align*}
E\left[\frac{TY}{e^*(X)}\right]-E\left[\frac{\left(1-T\right)Y}{\left(1-e^*(X)\right)}\right]&=E\left[E\left[\frac{TY}{e^*(X)}\bigg| X\right]\right]-E\left[E\left[\frac{\left(1-T\right)Y}{\left[1-e^*(X)\right]}\bigg| X\right]\right] \\
&=E\left[\frac{e(X)}{e^*(X)}\mu_1(X)\right]
-E\left[\frac{\left(1-e(X)\right)}{\left(1-e^*(X)\right)}\mu_0(X)\right].
\end{align*}

\noindent subtracting with $\Delta$ gives
\begin{align*}
\text{Bias}(\hat{\Delta}_{\text{\tiny{IPW}}_1}^*)= &E\left[\frac{e(X)}{e^*(X)}\mu_1(X)\right]
-E\left[\frac{1-e(X)}{1-e^*(X)}\mu_0(X)\right]-(\mu_1-\mu_0).\\
\end{align*}

\noindent The biases of Theorems \ref{biasIPW2} and \ref{biasDR} are derived similarly. 
\subsection{Comparisons} \label{app2}

To study the consequences of model misspecification for the estimators we compare each difference involving $\mu_1(X)$ and $\mu_0(X)$ separately. For example we study Bias$(\hat{\Delta}_{\text{\tiny{IPW}}_1}^*)$ 

\begin{align*}
\hat{\Delta}_{\text{\tiny{IPW}}_1}^*-\Delta\stackrel{p}{\longrightarrow}&E\left[\frac{e(X)}{e^*(X)}\mu_1(X)\right]
-E\left[\frac{1-e(X)}{1-e^*(X)}\mu_0(X)\right]-(\mu_1-\mu_0)\\
=&E\left[\frac{e(X)}{e^*(X)}\mu_1(X)\right]-\mu_1+\mu_0-E\left[\frac{1-e(X)}{1-e^*(X)}\mu_0(X)]\right]
\end{align*}

\noindent Inequalities concerning the biases are made with respect to the absolute values for two parts separately, e.g., for Bias$(\hat{\Delta}_{\text{\tiny{IPW}}_1}^*)$ we investigate
\begin{equation}\label{ped1}
\left|\text{Bias}_1(\hat{\Delta}_{\text{\tiny{IPW}}_1}^*)\right|=\left|E\left[\frac{e(X)}{e^*(X)}\mu_1(X)\right]-\mu_1\right|.
\end{equation}

\noindent Since $|-a|=|a|$ the second part is 

\begin{equation}\label{ped2}
\left|\text{Bias}_2(\hat{\Delta}_{\text{\tiny{IPW}}_1}^*)\right|<\left|E\left[\frac{1-e(X)}{1-e^*(X)}\mu_0(X)\right]-\mu_0\right|,
\end{equation}

\noindent and similarly for (\ref{2}). The conditions derived for the first part of the biases, (\ref{1}) and (\ref{2}), can be directly applied to the second part of the biases replacing $e(X)/e^*(X)$, with $(1-e(X))/(1-e^*(X))$ and $\mu_1(X)$ with $\mu_0(X)$. In a similar manner we have for Bias$_2(\hat{\Delta}_{\text{\tiny{DR}}}^*)$ 

\begin{equation*}
\left|E\left[\frac{(1-e(X))-
	(1-e^*(X))\left(\mu_{0}(X)-\mu^*_0(X)\right)}{\left(1-e^*(X)\right)}\right]\right|=\left|E\left[\frac{\left(e(X)
	-e^*(X)\right)\left(\mu_{0}(X)-\mu^*_0(X)\right)}{\left(1-e^*(X)\right)}\right]\right|,
\end{equation*}

\noindent and the conditions derived for the first part of the bias, (\ref{3}), can be directly applied to (\ref{ped2}) additionally replacing $\mu_1^*(X)$ with $\mu_0^*(X)$.
\vskip 0.5cm

\begin{proof}[Proof of Theorem \ref{dr2}]
	
	\vskip 0.5cm
	1. Assuming that $E\left[\left(\frac{e(X)}{e^*(X)}-1\right)\mu_1(X)\right]> 0$: 
	
	\begin{equation*}
	\left| E\left[\left(\frac{e(X)}{e^*(X)}-1\right)\left(\mu_1(X)-{\mu}^*_{1}(X)\right)\right]\right|<  	 E\left[\left(\frac{e(X)}{e^*(X)}-1\right)\mu_1(X)\right],
	\end{equation*} 
	
	\noindent Here, we have 
	
	\begin{align*}
	-E\left[\left(\frac{e(X)}{e^*(X)}-1\right)\mu_1(X)\right] &<\nonumber E\left[\left(\frac{e(X)}{e^*(X)}-1\right)\left(\mu_1(X)-{\mu}^*_{1}(X)\right)\right]< E\left[\left(\frac{e(X)}{e^*(X)}-1\right)\mu_1(X)\right],
	\end{align*}
	\noindent and
	\begin{equation}\label{cas}
	0 < E\left[\left(\frac{e(X)}{e^*(X)}-1\right)\mu_1^*(X)\right]<   	 2\cdot E\left[\left(\frac{e(X)}{e^*(X)}-1\right)\mu_1(X)\right].
	\end{equation}

	2. Assuming that $E\left[\left(\frac{e(X)}{e^*(X)}-1\right)\mu_1(X)\right]< 0$: 
	
	\begin{equation*}
	\left| E\left[\left(\frac{e(X)}{e^*(X)}-1\right)\left(\mu_1(X)-{\mu}^*_{1}(X)\right)\right]\right|<  	 -E\left[\left(\frac{e(X)}{e^*(X)}-1\right)\mu_1(X)\right],
	\end{equation*} 
	
	\noindent Here, we have 
	
	\begin{align}\label{case2}
	E\left[\left(\frac{e(X)}{e^*(X)}-1\right)\mu_1(X)\right] &<\nonumber E\left[\left(\frac{e(X)}{e^*(X)}-1\right)\left(\mu_1(X)-{\mu}^*_{1}(X)\right)\right]<- E\left[\left(\frac{e(X)}{e^*(X)}-1\right)\mu_1(X)\right],\\
	2\cdot E\left[\left(\frac{e(X)}{e^*(X)}-1\right)\mu_1(X)\right] &<  E\left[\left(\frac{e(X)}{e^*(X)}-1\right)\mu_1^*(X)\right]<   0	 
	\end{align}
	and by (\ref{cas}) and (\ref{case2}) 
	\begin{align*}
	\left|E\left[\left(\frac{e(X)}{e^*(X)}-1\right)\mu_1^*(X)\right]\right|\leq  2\cdot\left| E\left[\left(\frac{e(X)}{e^*(X)}-1\right)\mu_1(X)\right] \right|.
	\end{align*}
\end{proof}

\begin{proof}[Proof of Theorem \ref{dr3}]

	\noindent a) We have that 
	
	\begin{equation*}
	\left|cov\left[\frac{e(X)}{e^*(X)},\mu_1(X)\right]-cov\left[\frac{e(X)}{e^*(X)},\mu_1^*(X)\right]\right|=\left|E\left[\left(\frac{e(X)}{e^*(X)}-1\right)\left(\mu_1(X)-{\mu}^*_{1}(X)\right)\right]\right|.
	\end{equation*}
	\vskip 0.35cm
	by $\mu^*_1=\mu_1$. Further, assuming that $0<E\left[\frac{e(X)}{e^*(X)}-1\right] \mu_1<cov\left[\frac{e(X)}{e^*(X)},\mu_1^*(X)\right]<cov\left[\frac{e(X)}{e^*(X)},\mu_1(X)\right]$
	yields
	\begin{align*}&	0<cov\left[\frac{e(X)}{e^*(X)},\mu_1(X)\right]-cov\left[\frac{e(X)}{e^*(X)},\mu_1^*(X)\right]<cov\left[\frac{e(X)}{e^*(X)},\mu_1(X)\right]-E\left[\frac{e(X)}{e^*(X)}-1\right] \mu_1,
	\end{align*}
	\vskip 0.35cm
	\noindent from which it follows that
	
	\begin{equation*}
	\left|cov\left[\frac{e(X)}{e^*(X)},\mu_1(X)\right]-cov\left[\frac{e(X)}{e^*(X)},\mu_1^*(X)\right]\right|<\left|cov\left[\frac{e(X)}{e^*(X)},\mu_1(X)\right]+E\left[\frac{e(X)}{e^*(X)}-1\right] \mu_1\right|,
	\end{equation*}	
	\vskip 0.5cm
	\noindent which is equivalent to the desired result.
	
	\vskip 0.5cm
	b) see proof of Theorem \ref{dr2}.
	
\end{proof}

\begin{proof}[Proof of Theorem \ref{dr4}]
	\vskip 0.35cm
	1. Assuming that:\\
	$$cov\left[\frac{e(X)}{e^*(X)}, \mu_1(X)\right]> 0\iff\frac{E\left[\frac{e(X)}{e^*(X)}\mu_1(X)\right]}{E\left[\frac {e(X)}{e^*(X)}\right]}-\mu_1> 0$$

	\noindent Here, we have 
	\begin{footnotesize}
		\begin{align*}
		-\frac{E\left[\frac{e(X)}{e^*(X)}\mu_1(X)\right]}{E\left[\frac {e(X)}{e^*(X)}\right]}+\mu_1 &< E\left[\left(\frac{e(X)}{e^*(X)}-1\right)\left(\mu_1(X)-{\mu}^*_{1}(X)\right)\right]<\frac{E\left[\frac{e(X)}{e^*(X)}\mu_1(X)\right]}{E\left[\frac {e(X)}{e^*(X)}\right]}-\mu_1,\\\nonumber
		-cov\left[\frac{e(X)}{e^*(X)}, \mu_1(X)\right] &< E\left[\frac {e(X)}{e^*(X)}\right] E\left[\left(\frac{e(X)}{e^*(X)}-1\right)\left(\mu_1(X)-{\mu}^*_{1}(X)\right)\right]<cov\left[\frac{e(X)}{e^*(X)}, \mu_1(X)\right],\\\nonumber
		\end{align*}
	\end{footnotesize} and
	\begin{footnotesize}
		\begin{equation}\label{dr.a}
		E\left[\left(\frac{e(X)}{e^*(X)}-1\right)\mu_1(X)\right]-\frac{cov\left[\frac{e(X)}{e^*(X)}, \mu_1(X)\right]}{E\left[\frac {e(X)}{e^*(X)}\right]}< E\left[\left(\frac{e(X)}{e^*(X)}-1\right)\mu_1^*(X)\right]<E\left[\left(\frac{e(X)}{e^*(X)}-1\right)\mu_1(X)\right]+\frac{cov\left[\frac{e(X)}{e^*(X)}, \mu_1(X)\right]}{E\left[\frac {e(X)}{e^*(X)}\right]}
		\end{equation}
	\end{footnotesize}
	
	2. Assuming that:
	$$cov\left[\frac{e(X)}{e^*(X)}, \mu_1(X)\right]< 0\iff\frac{E\left[\frac{e(X)}{e^*(X)}\mu_1(X)\right]}{E\left[\frac {e(X)}{e^*(X)}\right]}-\mu_1< 0,$$ 
	\noindent then it follows that
	\begin{footnotesize}
		\begin{align*}
		\frac{E\left[\frac{e(X)}{e^*(X)}\mu_1(X)\right]}{E\left[\frac {e(X)}{e^*(X)}\right]}-\mu_1&< E\left[\left(\frac{e(X)}{e^*(X)}-1\right)\left(\mu_1(X)-{\mu}^*_{1}(X)\right)\right]<-\frac{E\left[\frac{e(X)}{e^*(X)}\mu_1(X)\right]}{E\left[\frac {e(X)}{e^*(X)}\right]}+\mu_1\\\nonumber
		cov\left[\frac{e(X)}{e^*(X)}, \mu_1(X)\right] &< E\left[\frac {e(X)}{e^*(X)}\right] E\left[\left(\frac{e(X)}{e^*(X)}-1\right)\left(\mu_1(X)-{\mu}^*_{1}(X)\right)\right]<-cov\left[\frac{e(X)}{e^*(X)}, \mu_1(X)\right],\\\nonumber
		\end{align*}
	\end{footnotesize}
	\noindent and
	\begin{footnotesize}
		\begin{equation}\label{dr.b}
		E\left[\left(\frac{e(X)}{e^*(X)}-1\right)\mu_1(X)\right]+\frac{cov\left[\frac{e(X)}{e^*(X)}, \mu_1(X)\right]}{E\left[\frac {e(X)}{e^*(X)}\right]}< E\left[\left(\frac{e(X)}{e^*(X)}-1\right)\mu_1^*(X)\right]<E\left[\left(\frac{e(X)}{e^*(X)}-1\right)\mu_1(X)\right]-\frac{cov\left[\frac{e(X)}{e^*(X)}, \mu_1(X)\right]}{E\left[\frac {e(X)}{e^*(X)}\right]},	
		\end{equation}
	\end{footnotesize}
	
	hence by \eqref{dr.a} and \eqref{dr.b} we conclude that 
	
	\begin{equation*}\scalebox{0.9}{$
		E\left[\left(\frac{e(X)}{e^*(X)}-1\right)\mu_1(X)\right]-\frac{\left|cov\left[\frac{e(X)}{e^*(X)}, \mu_1(X)\right]\right|}{E\left[\frac {e(X)}{e^*(X)}\right]}< E\left[\left(\frac{e(X)}{e^*(X)}-1\right)\mu_1^*(X)\right]<E\left[\left(\frac{e(X)}{e^*(X)}-1\right)\mu_1(X)\right]+\frac{\left|cov\left[\frac{e(X)}{e^*(X)}, \mu_1(X)\right]\right|}{E\left[\frac {e(X)}{e^*(X)}\right]},$}
	\end{equation*}
\end{proof}

\begin{proof} [Proof of Theorem \ref{dr5}]
	For a) assuming $cov\left[\frac{e(X)}{e^*(X)},\mu_1(X)\right]-\left|\frac{cov\left[\frac{e(X)}{e^*(X)},\mu_1(X)\right]}{E\left[\frac {e(X)}{e^*(X)}\right]}\right|<cov\left[\frac{e(X)}{e^*(X)}, \mu_1^*(X)\right]
	<cov\left[\frac{e(X)}{e^*(X)}, \mu_1(X)\right]+\left|\frac{cov\left[\frac{e(X)}{e^*(X)},\mu_1(X)\right]}{E\left[\frac {e(X)}{e^*(X)}\right]}\right|$ is equivalent to 
	\begin{equation*}
	\left|cov\left[\frac{e(X)}{e^*(X)},\mu_1(X)\right]-cov\left[\frac{e(X)}{e^*(X)},\mu_1^*(X)\right]\right|<\frac{\left|cov\left[\frac{e(X)}{e^*(X)}, \mu_1(X)\right]\right|}{E\left[\frac{e(X)}{e^*(X)}\right]}
	\end{equation*}
	
	\vskip 0.5cm
	further, $\mu^*_1=\mu_1$ implies
	\begin{equation*}
	\left|cov\left[\frac{e(X)}{e^*(X)},\mu_1(X)\right]-cov\left[\frac{e(X)}{e^*(X)},\mu_1^*(X)\right]\right|=\left|E\left[\left(\frac{e(X)}{e^*(X)}-1\right)\left(\mu_1(X)-{\mu}^*_{1}(X)\right)\right]\right|,
	\end{equation*}
	and also, 
	\begin{equation*}
	\frac{\left|cov\left[\frac{e(X)}{e^*(X)}, \mu_1(X)\right]\right|}{E\left[\frac{e(X)}{e^*(X)}\right]}=\left|\frac{E\left[\frac{e(X)}{e^*(X)}\mu_1(X)\right]}{E\left[\frac {e(X)}{e^*(X)}\right]}-\mu_1\right|,
	\end{equation*}	
	which establishes the result.
	\vskip 0.5cm
	For b) we use that if $E\left[\left(\frac{e(X)}{e^*(X)}-1\right){\mu}_{1}(X)\right]$ and $cov\left[\frac{e(X)}{e^*(X)}, \mu_1(X)\right]$ are both positive
	
	then,
	
	\begin{align*}
	E\left[\left(\frac{e(X)}{e^*(X)}-1\right){\mu}_{1}(X)\right]+\left|cov\left[\frac{e(X)}{e^*(X)}, \mu_1(X)\right]\right|&=\left|E\left[\left(\frac{e(X)}{e^*(X)}-1\right){\mu}_{1}(X)\right]+cov\left[\frac{e(X)}{e^*(X)}, \mu_1(X)\right]\right|\\
	\intertext{and}
	-\left|E\left\{\left[\frac{e(X)}{e^*(X)}-1\right]{\mu}_{1}(X)\right\}+cov\left[\frac{e(X)}{e^*(X)}, \mu_1(X)\right]\right|&<E\left[\left(\frac{e(X)}{e^*(X)}-1\right){\mu}_{1}(X)\right]+\left|cov\left[\frac{e(X)}{e^*(X)}, \mu_1(X)\right]\right|.
	\end{align*}
	If $E\left[\left(\frac{e(X)}{e^*(X)}-1\right){\mu}_{1}(X)\right]$ and $cov\left[\frac{e(X)}{e^*(X)}, \mu_1(X)\right]$ are both negative, then we have
	
	\begin{align*}
	E\left[\left(\frac{e(X)}{e^*(X)}-1\right){\mu}_{1}(X)\right]+\left|cov\left[\frac{e(X)}{e^*(X)}, \mu_1(X)\right]\right|&<\left|E\left[\left(\frac{e(X)}{e^*(X)}-1\right){\mu}_{1}(X)\right]+cov\left[\frac{e(X)}{e^*(X)}, \mu_1(X)\right]\right|,\\
	\intertext{and}
	-\left|E\left[\left(\frac{e(X)}{e^*(X)}-1\right){\mu}_{1}(X)\right]+cov\left[\frac{e(X)}{e^*(X)}, \mu_1(X)\right]\right|&=E\left[\left(\frac{e(X)}{e^*(X)}-1\right){\mu}_{1}(X)\right]+\left|cov\left[\frac{e(X)}{e^*(X)}, \mu_1(X)\right]\right|,
	\end{align*}
	and the necessary condition from Theorem \ref{dr4} follows.
\end{proof}

\end{document}